\documentclass[12pt]{article}
  \usepackage[centertags]{amsmath}
  \usepackage{amsfonts}
  \usepackage{amssymb}
  \usepackage{amsthm}
  \usepackage{newlfont}

  \newlength{\defbaselineskip}
  \newcommand{\setlinespacing}[1]%
                               {\setlenght{\baselineskip}{#1 \defbaselineskip}}

  \theoremstyle{plain}

  \newtheorem{thm}{Theorem}[section]
  
  \newtheorem{lem}[thm]{Lemma}
  
  \newtheorem{rem}[thm]{Remark}
  
  \theoremstyle{definition}
  \newtheorem{defi}[thm]{Definition}
  \newtheorem{exm}[thm]{Example}

\numberwithin{equation}{section}

\begin{document}
\begin{center}
{\bf  Bernstein's Lethargy Theorem in Fr\'{e}chet Spaces}
\end{center}
\vspace{.15 cm}
\begin{center}
\small{Asuman G\"{u}ven AKSOY  and Grzegorz LEWICKI}
\end{center}

  \date{September 26, 2014}
\mbox{~~~}\\
\mbox{~~~}\\
{\bf Abstract.} {\footnotesize In this paper we consider  Bernstein's Lethargy Theorem (BLT)  in the context of Fr\'{e}chet spaces. Let $X$ be an infinite-dimensional Fr\'echet space and let $\mathcal{V}=\{V_n\}$ be a nested sequence of subspaces of $ X$ such that $ \overline{V_n} \subseteq V_{n+1}$ for any $ n \in \mathbb{N}$ 
and 
$
X=\overline{ \bigcup_{n=1}^{\infty}V_n}.
$ 
Let $ e_n$ be a decreasing sequence of positive numbers tending to 0. Under an additional natural condition on $\sup\{\mbox{dist}(x, V_n)\}$, we prove that
  there exists $ x \in X$ and $ n_o \in \mathbb{N}$   such that
$$
\frac{e_n}{3} \leq \mbox{dist}(x,V_n) \leq 3 e_n
$$ 
 for any $ n \geq n_o$. By using the above theorem,  we  prove both Shapiro's \cite{Sha} and Tyuremskikh's \cite{Tyu}  theorems for  Fr\'{e}chet  spaces.  Considering  rapidly decreasing sequences, 
other versions of the BLT theorem in Fr\'{e}chet  spaces  will be discussed. We also give a theorem improving Konyagin's  \cite{Kon} result for Banach spaces.

\footnotetext{{\bf Mathematics Subject Classification (2000):}
41A25, 41A50, 41A65. \vskip1mm {\bf Key words: } Best  Approximation, Bernstein's Lethargy Theorem, Fr\'{e}chet Spaces.}

\section{ Introduction}
Let $C[0,1]$ denote the space of real valued continuous functions on $[0,1]$ with the supremum norm $|| . ||$. For $f \in C[0,1]$, the sequence of  best approximations  (or equivalently  the distance from $f$  to the linear subspace of polynomials $P_n$ of degree $\leq n$), are defined as: 
$$\mbox{dist}(f,P_n) =\rho(f, P_n)= \rho_n(f)= \inf \{ || f-p||:\,\,\, p\in P_n\}.$$

Clearly  $ \rho(f, P_1) \geq\rho(f,P_2)\geq \cdots $ and  $\{\rho(f,P_n)\}$ form a nonincreasing sequence of best approximations.  Furthermore we have, 
\renewcommand{\labelenumi}{\alph{enumi})}

\begin{enumerate}
\item $ \rho_n(\lambda f)= |\lambda| \rho_n(f)$
\item $\rho_n(f+v)= \rho_n(f)$ for $f\in C[0,1]$ and $v\in P_n$
\item $\rho_n(f_1+f_2)\leq \rho_n(f_1)+\rho_n(f_2)$ and $ |\rho_n(f_1)-\rho_n(f_2)| \leq \|f_1-f_2\|$ for $f_1,f_2 \in C[0,1]$. 
\end{enumerate}
The last property implies the continuity of the mapping from $C[0,1]$ to $\mathbb{R}^+$ defined as $f \rightarrow \rho_n(f)$.

The density of polynomials in $C[0,1]$ implies that $$\displaystyle \lim_{n \to \infty} \rho_n(f) = 0.$$ However, the Weierstrass approximation theorem gives no information about the speed of convergence for $\rho_n(f)$.   Bernstein, considered ``the inverse problem of the theory of best approximation" and showed that for each nonincreasing, null sequence $(d_n)$, there exist a function $f\in C[0,1]$ with $\rho_n(f)= d_n$, see  \cite{Ber}. This theorem is sometimes referred as Bernstein's Lethargy Theorem (BLT) and it has been applied to the theory of quasi analytic functions in several complex variables \cite{Ple2} and used in the constructive theory of functions \cite{Sin}.  Following the proof of Bernstein, Timan \cite{Tim} extended his result to an arbitrary system of strictly embedded \textit{finite dimensional} subspaces $Y_n$.  Later Shapiro, \cite{Sha}, replacing $C[0,1]$  with an arbitrary Banach space $(X, ||.||)$ and the sequence of $n$-dimensional subspaces of polynomials of degree $\leq n$ by a sequence $(Y_n)$ where $Y_1 \subset Y_2 \subset \cdots$ are strictly embedded \textit{closed subspaces} of $X$, showed that in this setting, for each null sequence $(d_n)$ of non-negative numbers, there is a vector $x\in X$ such that 
$$ \rho_n(x)=\inf \{||x-u||:\,\,\, u\in Y_n\} \neq O (d_n).$$
Thus there is no $M>0$ such that $\rho_n(x) \leq M d_n$. In other words $\rho_n(x)$ decay arbitrarily slowly. This result  was strengthened by Tyuremskikh \cite{Tyu} who established that the sequence of best approximations may converge to zero at an arbitrary slow rate:  for any expanding sequence $\{Y_n\}$ of subspaces and for any sequence $\{d_n\}$ of positive numbers converging to zero, he constructed an element $x\in X$  such that $\displaystyle\lim_{n \rightarrow \infty} \rho(x, Y_n) =0$ and $\rho(x, Y_n) \geq n$.  However, it is also possible that the sequence of best approximations may converge to zero arbitrarily fast. For example in \cite{Al-To}, under same conditions imposed on $\{Y_n\}$ and $\{d_n\}$, for any sequence $\{c_n\}$ of positive numbers converging to zero, there exists an element $x\in X$ such that $\displaystyle\lim_{ n \rightarrow \infty} \frac{\rho(x, Y_n)}{d_n} =0$ but $\displaystyle \frac{\rho(x, Y_n)}{d_n}  \neq  O (c_n).$ For a generalization of Shapiro's theorem we refer the reader to \cite{Al-Oi} and an application of Tyuremskikh's theorem to convergence of sequence of bounded linear operators consult \cite{De-Hu}. For other versions of Bernstein's Lethargy Theorem see \cite{Ak-Al}, \cite{Ak-Le},
\cite{Al-To}, \cite{Lew1}, \cite{Pli}.\\

These ideas bring us to the following question:
given a Banach space $X$, a strictly increasing sequence $\{Y_n \}$ of subspaces of $X$ with a dense union, and a non-increasing null sequence $\{d_n\} \subset  [0,\infty)$, does there exist $x \in X$ such that$\mbox{ dist}(x, Y_n) = d_n $ for each $n$?\\
The only known spaces $X$ in which the answer is always ``yes" are the Hilbert spaces, \cite{Tyu2}.  For a general (separable) Banach space $X$, a solution $x$ is known to exist whenever all $Y_n$  are finite-dimensional \cite{Tim}. 
Moreover, it is known that  if $X$ has  the above property, then it is reflexive, \cite{Tyu2}.   Later,  progress on this question was made by Borodin \cite{Br}. Borodin establishes the existence of such an element in case of rapidly decreasing sequences:
\begin{thm}[Borodin]
  \label{thm:Borodin}
  Let $X$ be an arbitrary infinite-dimensional Banach space, $Y_1 \subset Y_2
  \subset \cdots$ be an arbitrary countable system of strictly nested subspaces
  in $X$, and fix a numerical sequence $\{e_n\}_{n=1}^\infty$ such that $e_n >
  \sum_{k = n+1}^\infty e_k$ for all natural $n \ge n_0$ with $e_n > 0$. Then there is an element $x \in X$ such that $$\rho(x,Y_n) = e_n,$$  for $n =
  1,2,\ldots$.
\end{thm}
Konyagin \cite{Kon} improved the result of Borodin by proving the following theorem:
\begin{thm}[Konyagin]
\label{thm:Konyagin}
Let $X$ be a real Banach space, $Y_1 \subset Y_2
  \subset \cdots$  a sequence of strictly embedded closed linear  subspaces of $X$, and $d_1 \geq d_2 \geq \cdots$  a nonincreasing sequence converging to zero, then there exist an element $x\in X$ such that
the distance $\rho(x, Y_n)$ from $x$ to $Y_n$ satisfies the inequalities
  $$ d_n \leq \rho(x, Y_n) \leq 8d_n \,\,\,\mbox{for}\,\,\, n=1,2,\cdots $$
\end{thm}

The aim of this paper is to generalize Theorem \ref{thm:Konyagin} (see Theorem \ref{main} ) to the case of Fr\'{e}chet spaces. Since in general approximation theory for Fr\'{e}chet  spaces is quite different from  approximation theory in Banach spaces, it is natural to ask whether or not  Bernstein's Lethargy Theorem is true in this case.  There are several papers concerning Bernstein Lethargy Theorem  in the context of Fr\'{e}chet spaces; we refer the reader to \cite{Alb}, \cite{Ple1}, \cite{Lew} and \cite{Ak-Al}  for progresses  in this direction. Notice that a version of Bernstein's lethargy theorem for metrizable topological linear spaces was proved in \cite{Lew} and
Theorem 4.1 in \cite{Lew} has been applied to the theory of quasianalytic functions in modular function spaces. We refer the reader to \cite{Kozl}, \cite{Kozl-Lew} for results in this direction. 

The main result of this paper is  Theorem \ref{main}. By using Theorem \ref{main} we are able to prove both Shapiro's and Tyuremskikh's theorems for  Fr\'{e}chet  spaces  (see Theorem 2.11 and Theorem 2.12). Theorems 2.5, 2.6 and 2.14 are other versions of the BLT theorem for Fr\'{e}chet  spaces. We also give a theorem improving Konyagin's result for Banach spaces (see Theorem 2.18).

\begin{defi}

 $(X, \| \cdot \|)$ is called a \textit{Fr\'{e}chet  space}, if it is a metric linear space which is complete with respect  to 
 its F-norm $||.||_F$ giving the topology. As usual by an F-norm we mean that $||.||_F$   satisfies the following conditions, \cite{Rol}:
 \begin{enumerate}
 \item $||x||_F =0$ if and only if $x=0$,
 \item $||\alpha x||_F = ||x||_F$ for all real or complex  $\alpha$  with $||\alpha||=1$,
 \item $||x+y||_F \leq ||x||_F+||y||_F$,
 \end{enumerate}
 
 \end{defi}
Many Fr\'{e}chet spaces  $X$ can also be constructed using  a countable family of semi norms $||x||_k$ where $X$ is a complete space with respect to this family of semi-norms. For example a translation invariant complete metric inducing the topology on $X$ can be defined as $$d(x,y) = \displaystyle \sum_{k=0}^{\infty} 2^{-k} \displaystyle \frac{||x-y||_k}{1+||x-y||_k}\,\,\mbox{ for}\,\,\, x,y \in X.$$
Clearly, every Banach space is a Fr\'{e}chet space, and the other well known example of a Fr\'{e}chet space is the vector space $C^{\infty}[ 0,1]$ of all infinitely differentiable functions $f:[0,1] \to \mathbb{R}$ where the semi norms are $||f||_k = \sup \{|f^{k} (x) |: \,\, x\in [0,1]\}$. For more information about Fr\'echet spaces the reader is referred to \cite{Rol}.
\section{Results}
We start with three technical lemmas.
\begin{lem}
\label{compact}
Let $(X, \| \cdot \|)$ be a Fr\'{e}chet space, and let $G$ and $V$ be nonempty subsets of $X$. Assume that $G$ is a compact set. Then for any $ \delta >0$ there exists a finite set $ \{v_1, ..., v_n  \}
\subset V$ such that for any $g \in G $ 
$$
\| g -v_j\| \leq dist(g, V) + \delta 
$$
for some $ j \in \{1,...,n\}.$
\end{lem}
\begin{proof}
Fix $ \delta > 0$ and $g \in G.$ Choose $ v_g \in V$ such that $ \| g - v_g\| < dist(g,V) + \delta.$ Since the function $ f_g : G \rightarrow \mathbb{R}$ defined by 
$$
f_g(h) = dist(h,V) +\delta - \| h - v_g\| 
$$
is continuous, there exists an open neighborhood of $g,$ $ U_g \subset X$ such that $ f_g(h) >0 $ for any $h \in U_g.$
Notice that $ G \subset \bigcup_{g \in G} U_g. $ By the compactness of $G,$ $ G \subset \bigcup_{i=1}^n U_{g_i}.$ Hence for any $ h \in G$ 
$$
\| h - v_{g_i}\| < dist(g,V) + \delta
$$
for some $ i \in \{ 1,...,n\},$ which completes the proof.
\end{proof}
\begin{lem}
\label{series}
Fix $ \epsilon > 0.$ Let $ (e_n)$ be a sequence of nonnegative real numbers such that for any $ n \in \mathbb{N}$ 
$$ 
e_n \geq (2 + \epsilon)e_{n+1}.
$$ 
Set for $ n \in \mathbb{N}$ $ b_n = \sum_{j=n}^{\infty} 2^{j-n}e_j.$ 
Then $ \sum_{n=1}^{\infty} b_n < + \infty.$
\end{lem}
\begin{proof}
Notice that by d'Alambert's criterion $ \sum_{n=1}^{\infty}2^n e_n < \infty .$ Hence for any $ n \in \mathbb{N},$ $b_n \in \mathbb{R}.$
Observe that
$$
\frac{b_{n+1}}{b_n} = \frac{ \lim_m \sum_{j=n+1}^{m+1} 2^{j-n-1}e_j}{\lim_m \sum_{j=n}^{m} 2^{j-n}e_j}
$$ 
and for $ m \geq n,$
$$
\frac{\sum_{j=n+1}^{m+1} 2^{j-n-1}e_j}{\sum_{j=n}^{m} 2^{j-n}e_j} \leq \frac{\sum_{j=n}^{m} 2^{j-n}e_j}{(2+\epsilon)\sum_{j=n}^{m} 2^{j-n}e_j} =\frac{1}{2+\epsilon}
$$
which completes the proof.
\end{proof}
\begin{lem}
\label{three}
Let $ \{ e_n\}$ be a sequence of positive numbers such that $ e_n \geq 3e_{n+1}$ for any $ n \in \mathbb{N}.$ The for any $m,n \in \mathbb{N} \setminus \{0\},$ 
$$
e_n\geq \sum_{j=n+1}^{n+m} 2^{j-n-1} e_j + 2^{m}e_{n+m}
$$
\end{lem}
\begin{proof}
It is enough to show this lemma for $ n=1.$ We proceed by the induction argument with respect to $m.$ If $m=1,$ then
$$
e_1 \geq 3e_2 = e_2 + 2e_2, 
$$
which shows our claim in this case.
Now assume that 
$$
e_1\geq \sum_{j=2}^{1+m} 2^{j-2} e_j + 2^{m}e_{1+m}.
$$
Observe that 
$$
e_1\geq \sum_{j=2}^{1+m} 2^{j-2} e_j + 2^{m}e_{1+m} \geq \sum_{j=2}^{1+m} 2^{j-2} e_j + 3(2^{m}e_{2+m})
$$
$$
= \sum_{j=2}^{1+m} 2^{j-2} e_j + 2^{m}e_{2+m} + 2^{m+1}e_{2+m} 
= \sum_{j=2}^{1+(m+1)} 2^{j-2} e_j + 2^{m+1}e_{1+(m+1)}, 
$$
as required. The proof is complete.
\end{proof}
Now let $X$ be an infinite dimensional Fr\'echet space equipped with an F-norm $\| \cdot \|.$ Assume that $\mathcal{V} = \{ V_n\}$ is a nested sequence of linear subspaces of $X$ satisfying 
$\overline{V_n} \subseteq V_{n+1},$ where the closure is taken with respect to $ \| \cdot \| .$ Let for $ n \in \mathbb{N}$ 
\begin{equation}
 \label{Fn}
F_n = \{ v \in V_{n+1} : dist(v,V_j)= e_j \hbox{ for } j=1,...,n\}. 
\end{equation}
The first step in obtaining our version of Bernstein's Lethargy Theorem is to show that under some assumptions for any $ n \in \mathbb{N},$ $F_n$ is not only nonempty (by the argument of compactness 
it is sufficient to prove Bernstein's Lethargy Theorem in finite-dimensional setting) but it also consists of elements of a special shape. This will be done in the next lemma.
Before presenting it, we introduce some notation. Set for $ x \in X$ 
\begin{equation}
\label{dist} 
\rho_n(x) = dist(x,V_n)
\end{equation}
and  define 
\begin{equation}
\label{dn}
d_{n, \mathcal{V}} = \sup \{ \rho_n(v) : v \in V_{n+1}\}.
\end{equation}
Let $ \{e_n\}$ be a decreasing sequence of positive real numbers satisfying
$$
\sum_{j=n}^{\infty} 2^{j-n} e_j < \min \{d_{n,\mathcal{V}}, e_{n-1}\}
$$ 
for any $ n \in \mathbb{N}$. (We put $e_o = \infty$ and $\sum_{j=1}^{\infty} 2^{j-1} e_j < \infty , $ if $ d_{1, \mathcal{V}} = \infty.)$
Fix a sequence $ (\delta_n)$ of positive numbers such that for any $ n \in \mathbb{N}$
\begin{equation}
\label{Cauchy}
\sum_{j=n}^{\infty} 2^{j-n} (e_j + \delta_j) < \min \{e_{n-1}, d_{n, \mathcal{V}}\}
\end{equation}
and $ v_n \in V_{n+1} \setminus V_n$ such that 
\begin{equation}
\label{vn}
\sum_{j=n}^{\infty} 2^{j-n} (e_j  + \delta_j)= \rho_n(v_n)\geq \|v_n\| - \delta_n.
\end{equation}
Since the function $ t \rightarrow \rho_n(tx)$ is continuous for any $ x\in X,$ $ \rho_n(x) = \rho_n(x-v)$ for any $ x \in X$ and $v\in V_n,$ such $v_n$ exist by \ref{Cauchy} and  the definition of infimum. 
Then we can state:
\begin{lem}
\label{crucial}
Let $ (X, \| \cdot \|)$ be an infinite-dimensional Fr\'{e}chet space. Let $ \{V_n\} =\mathcal{V}$ be a nested sequence of linear subspaces of $X$ satisfying 
$\overline{V_n} \subseteq V_{n+1},$ where the closure is taken with respect to $ \| \cdot \| .$  Let $ \{e_n\}$ be a decreasing sequence of positive real numbers satisfying \ref{Cauchy} with a fixed sequence of positive numbers $ \{ \delta_n\}.$ 
Then for any $ n \in \mathbb{N}$ there exists a finite set $ Z_n \subset V_n$ and $ w_n \in F_n,$ such that
$$
w_n = \sum_{j=1}^n q_{j,n},
$$
where $ q_{j,n} = t_{j,n}v_j- z_{j,n},$ $ t_{j,n} \in [0,1],$ $ z_{j,n} \in Z_j$ for $j=1,...,n$ and $ v_j$ are given by \ref{vn}.  Moreover,
$$
\| q_{j,n}\| < \sum_{l=j}^{\infty}2^{l-j}(e_l+\delta_l).
$$
\end{lem}
\begin{proof}
Notice that for any $ n \in \mathbb{N}$ the set $ G_n = \{ tv_n: t \in [0,1] \} $ is compact. By Lemma \ref{compact}, applied to $ G_n$ and $V_n,$ we can find a finite set $ Z_n\subset V_n$ such that for any 
$g \in G_n$
\begin{equation}
\label{shift}
\| g- z_n\| < \rho_n(g) + \delta_n
\end{equation}
for some $ z_n \in Z_n.$
Fix  $ n \in \mathbb{N}.$ For $ k \in \{ 0,...,n-1\}.$ Let 
$$
E_{k,n} = \{ v \in V_{n+1} : \rho_j(v)= e_j \hbox{ for } j=n,...,n-k\}.
$$
Now we show that there exists $ q_{n,n},...q_{1,n}, $ such that 
$ q_{j,n} = t_{j,n}v_j- z_{j,n},$ where $ t_{j,n} \in [0,1]$ and $ z_{j,n} \in Z_j$ for $j=1,...,n$ with
\begin{equation}
\label{estimate} 
\sum_{l=j}^n \| q_{l,n}\| < \sum_{l=j}^n 2^{l-j} (e_l+\delta_l)
\end{equation}
for  $j=1,...,n$
satisfying
\begin{equation}
\label{nonempty}
w_{k,n} = \sum_{j=n-k}^n q_{j,n} \in E_{k,n}
\end{equation}
for any $k \in \{ 0,...,n-1\}.$
First we construct $ q_{n,n}.$ Notice that the function $ f_n(t) = \rho_n(tv_n) $ is continuous and by our assumptions on $ v_n,$ 
$$
f_n(1) = \sum_{j=n}^{\infty} 2^{j-n} (e_j+\delta_j) > e_n \hbox{ and }  f_n(0) =0.
$$
By the Intermediate Value Theorem there exists $ t_{n,n} \in [0,1]$ such that 
\begin{equation} 
\label{firststep}
\rho_n(t_{n,n}v_n) = e_n.
\end{equation} 
Again by the continuity of $f_n$ we can assume that $ t_{n,n}$ is the smallest number satisfying \ref{firststep}. By \ref{shift} we can choose $z_n \in Z_n$ such that 
$ \|t_{n,n}v_n - z_n\| < \rho_n(t_{n,n}v_n) + \delta_n.$ Let 
\begin{equation}
\label{qnn} 
q_{n,n}= t_{n,n}v_n - z_n.
\end{equation}
Observe that $ \rho_n(q_{n,n}) =e_n,$ since $ z_n \in V_n$ and $V_n$ is a linear subspace. Hence 
$$
w_{o,n} = q_{n,n} \in E_{o,n}.
$$  
Also $ q_{n,n}$ satisfies \ref{estimate}. Now fix $ k \in \{0,...,n-2\}.$
Assume that we have constructed $ q_{n,n},...,q_{n-k,n}$ satisfying \ref{nonempty} and \ref{estimate}. Now we construct $ q_{n-k-1,n}.$ By \ref{Cauchy} applied to $ n-k-1$ , we can choose $s_{n-k-1} \in [0,1]$ (the smallest one)
such that 
$$
\rho_{n-k-1}(s_{n-k-1}v_{n-k-1}) = e_{n-k-1}+ \sum_{l=n-k}^n 2^{l-n+k}(e_l + \delta_l).
$$
Let $ u_k = \sum_{l=n-k}^n q_{l,n}.$
Observe that
$$
e_{n-k-1}+ \sum_{l=n-k}^n 2^{l-n+k}(e_l + \delta_l)=\rho_{n-k-1}(s_{n-k-1}v_{n-k-1}) 
$$
$$
\leq \rho_{n-k-1}(s_{n-k-1}v_{n-k-1}+u_k) + \rho_{n-k-1}(u_k) 
$$
$$
\leq \rho_{n-k-1}(s_{n-k-1}v_{n-k-1}+u_k) + \|u_k\| 
$$
$$
\leq \rho_{n-k-1}(s_{n-k-1}v_{n-k-1}+u_k) + \sum_{l=n-k}^n 2^{l-n+k}(e_l + \delta_l).
$$
Hence
$$
\rho_{n-k-1}(s_{n-k-1}v_{n-k-1}+u_k) \geq e_{n-k-1}. 
$$
Let $ f_k(t) = \rho_{n-k-1}(tv_{n-k-1}+u_k).$ Observe that
$$ 
f_k(s_{n-k-1}) = \rho_{n-k-1}(s_{n-k-1}v_{n-k-1}+u_k) \geq e_{n-k-1}.
$$ 
Also by our assumptions
$$
f_k(0) = \rho_{n-k-1}(u_k) \leq \|u_k\| \leq \sum_{j=n-k}^n \| q_{j,n}\| 
$$
$$
< \sum_{l=n-k}^n 2^{l-n+k} (e_l+\delta_l) < \sum_{l=n-k}^{\infty} 2^{l-n+k} (e_l+\delta_l) < e_{n-k-1}.
$$
Since the function $f_k$ is continuous,  by the Intermediate Value Theorem  there exists $t_{n-k-1} \in [0,s_{n-k-1}]$ such that 
$$
\rho_{n-k-1}(t_{n-k-1}v_{n-k-1}+u_k) = e_{n-k-1}.
$$
By \ref{shift} we can choose $z_{n-k-1} \in Z_{n-k-1}$ such that 
$$
 \|t_{n-k-1}v_{n-k-1} - z_{n-k-1}\| < \rho_{n-k-1}(t_{n-k-1}v_{n-k-1}) + \delta_{n-k-1}.
$$  
Let 
\begin{equation}
\label{q(n-k-1,n)} 
q_{n-k-1,n}= t_{n-k-1}v_{n-k-1} - z_{n-k-1}.
\end{equation}
Since $ q_{n-k-1,n} \in V_{n-k}$ and $ z_{n-k-1} \in Z_{n-k-1} \subset V_{n-k-1},$ by the above construction
$$
w_{k-1,n} = \sum_{j=n-k-1}^n q_{j,n} \in E_{k+1,n},
$$
which shows that the equation \ref{nonempty} is satisfied.
Moreover by the choice of $ s_{n-k-1}, $ 
$$
\| q_{n-k-1,n}\| = \|t_{n-k-1,n}v_{n-k-1} - z_{n-k-1}\| 
$$
$$
< \rho_{n-k-1}(t_{n-k-1}v_{n-k-1}) + \delta_{n-k-1} \leq \rho_{n-k-1}(s_{n-k-1}v_{n-k-1}) + \delta_{n-k-1} 
$$
$$
= \delta_{n-k-1} + e_{n-k-1}+ \sum_{l=n-k}^n 2^{l-n+k}(e_l + \delta_l).
$$
Observe that 
$$
\sum_{l=n-k-1}^n \| q_{l,n}\| = \| q_{n-k-1,n}\| + \sum_{l=n-k}^n \| q_{l,n}\|
$$
$$ 
\leq \delta_{n-k-1} + e_{n-k-1}+ \sum_{l=n-k}^n 2^{l-n+k}(e_l + \delta_l) + \sum_{l=n-k}^n 2^{l-n+k} (e_l+\delta_l)
$$
$$
=\sum_{l=n-k-1}^n 2^{l-n+k+1}(e_l + \delta_l),
$$
which shows that \ref{estimate} is satisfied. Hence we have constructed $q_{1,n},...,q_{n,n}$ satisfying the requirements of our lemma. In particular, the above reasoning shows that $ E_{n-1,n}$ is nonempty and
$w_{n-1,n} = \sum_{j=1}^n q_{j,n} \in E_{n-1,n}.$
Since $F_n= E_{n-1,n},$ the proof is complete.
\end{proof}
 Applying Lemma \ref{crucial} we now show a version of Bernstein's Lethargy Theorem in Fr\'echet spaces for rapidly decreasing sequences $\{ e_n\}.$ 
\begin{thm}
\label{rapidlydecr1}
Let $ X, $ $ \mathcal{V},$ $ \{d_{n, \mathcal{V}}\}$ and $\{e_n\}$ be such as in Lemma \ref{crucial}. Then, there exists $ x \in X$ such that $ dist(x,V_n)= e_n$ for any $n \in \mathbb{N}.$ 
\end{thm}
\begin{proof}
Let for any $n \in \mathbb{N}$ 
$$
F_n = \{ v \in V_{n+1} : \rho_j(v)= e_j \hbox{ for } j=1,...,n\}.
$$
Take $ w_n = \sum_{j=1}^n q_{j,n} \in F_n$ constructed in Lemma \ref{crucial}. Fix $j \in \mathbb{N}.$ Observe that for any $n \in \mathbb{N}$ 
$ q_{j,n} = t_{j,n}v_j - z_{j,n}, $ where $ t_{j,n} \in [0,1]$ and $ z_{j,n} \in Z_j$ (see \ref{shift}). Since $Z_j$ is a finite set, we can select a subsequence $\{n_k\}$  and $ q_j \in V_{j+1}$ such that 
$$
\|q_{j,n_k} - q_j \| \rightarrow 0.
$$
Applying a diagonal argument we can choose a subsequence $ \{n_k\}$ such that for any $ j \in \mathbb{N},$
$$
\|q_{j,n_k} - q_j \| \rightarrow 0.
$$
Let $ s_k = \sum_{j=1}^k q_j.$ We show that $ \{s_k\}$ is a Cauchy sequence. Fix $ \epsilon > 0. $ By \ref{Cauchy} we can find $ l \in \mathbb{N} $ such that 
$$
\sum_{j=l}^{\infty} 2^{j-l}e_j < \epsilon.
$$
Notice that by \ref{Cauchy} and Lemma \ref{crucial} for $ n >k \geq l+1,$ 
$$
\| s_k -s_n\| = \| \sum_{j=k+1}^n q_j \| \leq \sum_{j=k+1}^n \| q_j \| 
$$
$$
\leq \sum_{j=k+1}^n (\sum_{m=j}^{\infty} 2^{m-j}(e_m+ \delta_m)) \leq \sum_{j=k}^{n-1} e_j \leq \epsilon, 
$$
which affirms our claim. Since $X$ is complete, $ \| s_k - x \| \rightarrow 0, $ where 
$$
x = \sum_{j=1}^{\infty} q_j. 
$$
Now we show that $ \| w_{n_k} - x \| \rightarrow 0.$  To do that, fix $ \epsilon > 0$ and $ n_o \in \mathbb{N} $ such that 
$$
\sum_{j=n_o}^{\infty} 2^{j-n_o}e_j < \epsilon.
$$
Note that for \ref{Cauchy} and Lemma \ref{crucial}, $ n_k \geq n_o +1,$ and
$$
\| w_{n_k} - x\| \leq \| \sum_{j=1}^{n_o} q_{j,n_k}- q_j\| + \| \sum_{j=n_o+1}^{n_k} q_{j,n_k}- \sum_{j=n_o+1}^{\infty} q_j\|
$$
$$
\leq \sum_{j=1}^{n_o} \|q_{j,n_k}- q_j\| + \sum_{j=n_o+1}^{\infty} \|q_{j,n_k}\| + \sum_{j=n_o+1}^{\infty} \| q_j\|
$$
$$
\leq \sum_{j=1}^{n_o} \|q_{j,n_k}- q_j\| + 2(\sum_{j=n_o+1}^{\infty}(\sum_{l=j}^{\infty}2^{l-j}(e_l+\delta_l))) \leq \| \sum_{j=1}^{n_o} q_{j,n_k}- q_j\| + 2\sum_{j=n_o}^{\infty} e_j 
$$
$$
\leq \sum_{j=1}^{n_o} \|q_{j,n_k}- q_j\| + 2 \epsilon.
$$
Since $\|q_{j,n_k} - q_j \| \rightarrow_k 0$ for any $ j \in \mathbb{N},$ 
$$
\sum_{j=1}^{n_o} \|q_{j,n_k}- q_j\| + 2 \epsilon \leq 3\epsilon
$$
for k sufficiently large, which shows that $ \| w_{n_k} - x \| \rightarrow 0.$
\newline
Now we show that $ \rho_j(x) = e_j$ for $ j \in \mathbb{N}.$ Fix $ j_o \in \mathbb{N}.$ Then for $ k \geq k_o $ $n_k \geq j_o.$ Since $ w_{n_k} \in F_{n_k}, $ 
$$
\rho_{j_o}(w_{n_k}) = e_{j_o} \hbox{ for } k \geq k_o.
$$
Hence 
$$
\rho_{j_o}(x)= \lim_k \rho_{j_o}(w_{n_k}) = e_{j_o},
$$
which completes the proof.
\end{proof}
\begin{thm}
\label{rapidlydecr2}
Let $ X, $ $ \mathcal{V},$ and $ \{d_{n,\mathcal{V}} \}$ be such as in Lemma \ref{crucial}. Assume that $ \{ e_n \} $ is a decreasing sequence of nonegative numbers satisfying 
$e_n < d_{n, \mathcal{V}}$ and $e_n \geq 3 e_{n+1}$ for any $ n \in \mathbb{N}.$
Then there exists $ x \in X$ such that $ dist(x,V_n)= e_n$ for any $n \in \mathbb{N}.$ 
\end{thm}
\begin{proof}
To prove our result, 
it is sufficient to verify if the assumptions of Theorem \ref{rapidlydecr1} are satisfied. But this is a consequence of Lemma \ref{series} and Lemma \ref{three}. 
\end{proof}
Now we show that Theorem \ref{rapidlydecr2} can be applied to prove a version of Bernstein's Lethargy Theorem in the case of any decreasing sequence $ \{e_n\}$ of positive numbers tending to 0 and any nested sequence $\{V_n\}$  
of linear subspaces of a Fr\'echet space $X,$ which is the main result of this paper.
To do that, being inspired by \cite{Kon}, we need the following construction of a decreasing sequence $\{ f_k\}$ of positive numbers tending to zero and a subsequence $ \{n_k\}.$ Put $ f_1 = e_1$ and $ n_1 =1.$ If $ e_1 \geq 3 e_2,$ then we define $ n_2 =2 $ and $ f_2 = e_2.$
If $ e_1 < 3e_2$ then $ f_2 =  \displaystyle \frac{f_1}{3}$ and 
$$
n_2 = \max\{ n \geq 2: f_1 \leq 3e_n\}.
$$
Since $ f_1=e_1 >0$ and $ e_n \rightarrow 0, $ $n_2$ is well-defined.
Now assume that we have constructed positive numbers $ f_1,...,f_k$ and $ n_1,...,n_k\in \mathbb{N}.$ We will construct $ f_{k+1}$ and $ n_{k+1}.$ If 
$$
e_{n_k+1} \leq \frac{f_k}{3}
$$
then we define $ f_{k+1} = e_{n_k+1} $ and $n_{k+1} = n_{k}+1.$ In the opposite case $ f_{k+1} = \displaystyle \frac{f_k}{3}$ and 
$$
n_{k+1} = \max\{ n \geq n_{k}+1: f_k \leq 3e_n\}.
$$ 
Since $ f_k >0$ and $ e_n \rightarrow 0, $ $n_{k+1}$ is well-defined.
\begin{exm}
\label{illustrative}
In the above discussion, if we take $ e_n = \displaystyle \frac{1}{n},$ then $f_k =\displaystyle  \frac{1}{3^{k-1}}$ and $n_k = 3^{k-1}.$ 
\end{exm}
\begin{lem}
\label{Konyagin}
Let $ \{ e_n\}$ be a decreasing sequence of positive numbers tending to $0.$ Let $ \{ f_k \} $ and $ \{ n_k\}$ be as in the above construction. Then for any $ k \in \mathbb{N},$ 
$f_k \geq 3f_{k+1}$ and $e_{n_k +1} \leq f_k.$
\end{lem}
\begin{proof}
Fix $ k \in \mathbb{N}.$  If $f_{k+1} = e_{n_k+1}, $ then by our construction 
$$
f_{k+1} = e_{n_k+1} \leq \frac{f_k}{3}.
$$
In the opposite case $f_{k+1} = \displaystyle \frac{f_k}{3},$ which validates our claim.
\newline
Analogously, if $f_{k+1} = e_{n_k+1},  $ then 
$$
e_{n_k+1} \leq \frac{f_k}{3} < f_k.
$$
If $f_{k+1} \neq e_{n_k+1} $ and $ f_{k} \neq e_{n_{k-1}+1},$ then by definition of $n_{k}, $ $ 3e_{n_k+1} \leq f_{k-1} = 3f_k.$ Finally, if $f_{k+1} \neq e_{n_k+1} $ and $f_{k} = e_{n_{k-1}+1} ,$  then 
$$
e_{n_k+1} \leq e_{n_{k-1}+1} = f_k,
$$
since the sequence $ \{ e_n\}$ is decreasing.
\end{proof}
Now we are ready to state the main result of this paper.
\begin{thm}
\label{main}
Let $X$ be a infinite-dimensional Fr\'echet space and let $\mathcal{V}=\{V_n\}$ be a nested sequence of subspaces of $ X$ such that $ \overline{V_n} \subseteq V_{n+1}$ for any $ n \in \mathbb{N}$ 
and 
$$
X=\overline{\displaystyle \bigcup_{n=1}^{\infty}V_n}.
$$ 
Let $ e_n$ be a decreasing sequence of positive numbers tending to 0. Assume that 
\begin{equation}
 \label{nondegenerate}
d_{\mathcal{V}} = \inf\{ d_{n,\mathcal{V}}: n \in \mathbb{N} \} > 0.
\end{equation}
Then there exists $ n_o \in \mathbb{N}$ and $ x \in X$ such that for any $ n \geq n_o$
$$
\frac{e_n}{3} \leq dist(x,V_n) \leq 3 e_n.
$$ 
\end{thm}
\begin{proof}
Let $ \{f_k\} $ and $\{ n_k\} $ be two sequences associated with $ \{ e_n\}$ by our construction. Set for $ k \in \mathbb{N}$, $ W_k = V_{n_k}$, and $ \mathcal{V}_1 = \{ W_k\}.$ Since for any $ k \in \mathbb{N}, $
$ V_{n_k +1} \subset W_{k+1},$ 
$$
d_{k, \mathcal{V}_1} \geq d_{n_k, \mathcal{V}} \geq d_{\mathcal{V}} >0.
$$
 By Lemma \ref{Konyagin}  $ f_k \geq 3f_{k+1}$ for any $k \in \mathbb{N}.$ Fix $ k_o \in \mathbb{N}$ such that for $ k \geq k_o$
$$
f_k < d_{\mathcal{V}} \leq d_{k, \mathcal{V}_1}.
$$ 
Applying Theorem \ref{rapidlydecr2} to $\{ f_k : k \geq k_o\}$ and $ \{ W_k : k \geq k_o\}, $  there exists $x \in X$ such that 
$$
dist(x, V_{n_k}) = f_k \hbox{ for } k \geq k_o.
$$ 
Let $ n_o = n_{k_o}+1.$ Fix $ n \geq n_o.$ Then there exists exactly one $ k \in \mathbb{N}$ such that $ n_k < n \leq n_{k+1}.$ If $ f_{k+1}=e_{n_k+1}, $ 
then by our construction $n_{k+1} = n_k +1$ and 
$$
\frac{e_n}{3} < e_n= e_{n_k+1}= f_{k+1} = dist(x,V_n) < 3e_n. 
$$
If $f_{k+1} \neq e_{n_k+1} ,$ then by Lemma \ref{Konyagin} and our assumptions
$$
dist(x,V_n) \geq dist(x,V_{n_{k+1}}) = f_{k+1} = \frac{f_k}{3} \geq \frac{e_{n_k+1}}{3} \geq \frac{e_n}{3}.
$$
Also if $f_{k+1} \neq e_{n_k+1},$ then, by definition of $ n_{k+1},$ 
$$
dist(x,V_n) \leq dist(x,V_{n_k}) = f_k \leq 3 e_n.
$$
The proof is complete.
\end{proof}
\begin{rem}
If $d_{\mathcal{V}}=0$ the above result holds true with the same proof for sequences $ \{ e_n\} $ satisfying $ e_n < d_{n,\mathcal{V}}$  for $ n \geq n_o.$
\end{rem}
From Theorem \ref{main} we can easily obtain a version of Shapiro's  theorem, \cite{Sha}, and a version of Tyuremskikh's  theorem, \cite{Tyu}, for Fr\'echet spaces.
\begin{thm}
Let the assumptions of Theorem \ref{main} be satisfied. Then there exists $ x \in X$ such that $\rho_n(x) \neq O(e_n).$ 
\end{thm}
\label{Sha}
\begin{proof}
By Theorem \ref{main} applied to the sequence $ \{ \sqrt{e_n}\}$ there exists $ x\in X$ and $ n_o \in \mathbb{N}$ such that
$$
\frac{\sqrt{e_n}}{3} \leq \rho_n(x) \leq 3 \sqrt{e_n}
$$ 
for $n \geq n_o.$
Since $ e_n \rightarrow 0,$ it is obvious that $ \rho_n(x) \neq O(e_n).$
\end{proof}
\begin{thm}
\label{Tyu}
Let the assumptions of Theorem \ref{main} be satisfied. Then there exists $ x \in X$ and $ n_o \in \mathbb{N}$ such that $\rho_n(x) \geq e_n$ for $ n\geq n_o.$ 
\end{thm}
\begin{proof}
By Theorem \ref{main} applied to the sequence $ \{ 3\sqrt{e_n}\}$ there exists $ x\in X$ and $ n_o \in \mathbb{N}$ such that
$$
\sqrt{e_n} \leq \rho_n(x) \leq 9 \sqrt{e_n}.
$$ 
for $ n\geq n_o.$
Since $ e_n \rightarrow 0,$ it is obvious that $ \rho_n(x) \geq \sqrt{e_n} \geq e_n$ for $ n\geq n_o.$ 
\end{proof}
Note that the parameter $d_{\mathcal{V}}$ defined by \ref{nondegenerate} can be equal to 0 which is illustrated in the next example. 
\begin{exm}
Let $X$ be a space of complex sequences with the Fr\'{e}chet norm of $x\in X$  given by $\|x\| = \displaystyle \sum_{i=0}^{\infty} 2^{-i} \displaystyle \frac{|x_i|}{1+|x_i|}. $  Set for $n\in \mathbb{N}$,  $ V_n =\{ x\in X: \,\,\, x_j=0\,\, \mbox{for}\,\, j\geq n \}. $
It is easy to see that 
$$
d_{n,\mathcal{V}} = \sup \{ \rho_n(v) : v \in X\} = \frac{1}{2^{n}}
$$
and consequently that  $d_{\mathcal{V}} = 0.$ 
\end{exm}
However, if a Fr\'echet space $ (X, \| \cdot \|)$ is equipped with an $s$-convex $F$-norm for some $s \in (0,1],$  which means 
\begin{equation}
\label{sconvex}
\| tx\| = |t|^s \| x\| \hbox{ for any } x \in X,
\end{equation}
then obviously for any linear subspace $V$ of $X$ and $ x \in X, $ $ dist(tx, V) = |t|^s dist (x,V).$ 
In particular this equality holds true in Banach spaces with $ s=1.$
Hence we can state the following theorem. 
\begin{thm}
\label{main1}
Let $X$ be a Fr\'echet space with an $s$-convex $F$-norm and let $\{V_n\}$ be a nested sequence of subspaces of $ X$ such that $ \overline{V_n} \subseteq V_{n+1}$ for any $ n \in \mathbb{N}.$ Let $ e_n$ be a decreasing sequence of positive numbers tending to 0.
Then there exists $ x \in X$ such that for any $ n \in \mathbb{N}$
$$
\frac{e_n}{3} \leq dist(x,V_n) \leq 3 e_n.
$$ 
\end{thm}
\begin{proof}
It is easy to see that by \ref{sconvex} $d_{\mathcal{V}} = \infty .$ Hence reasoning as in Theorem \ref{main}  the proof is complete.
\end{proof}
In particular we have:
\begin{thm}
\label{Banach}
Let $X$ be a Banach space and let $\{V_n\}$ be a nested sequence of subspaces of $ X$ such that $ \overline{V_n} \subseteq V_{n+1}$ for any $ n \in \mathbb{N}.$ Let $ e_n$ be a decreasing sequence of positive numbers tending to 0.
Then there exists $ x \in X$ such that for any $ n \in \mathbb{N}$
$$
\frac{e_n}{3} \leq dist(x,V_n) \leq 3 e_n.
$$ 
\end{thm}
In the case of Banach spaces, applying Theorem \ref{thm:Borodin} instead of Theorem \ref{rapidlydecr2} we can improve Theorem \ref{thm:Konyagin}.
We need the following lemma:
\begin{lem}
\label{two}
Let $ \{ e_n\}$ be a sequence of positive numbers such that $ e_n \geq 2e_{n+1}$ for any $ n \in \mathbb{N}.$ The for any $n \in \mathbb{N} \setminus \{0\},$ 
$$
e_n\geq \sum_{j=n+1}^{\infty} e_{j}
$$
\end{lem}
\begin{proof}
It is enough to show that for any $ n \in \mathbb{N}, $ $ e_1 \geq \sum_{j=2}^n e_n + e_n.$
If $ n =2, $ then $ e_1 \geq 2e_2 = e_2 + e_2, $ which shows our claim for $ n=2.$ The general case can be easily obtained by an induction argument with respect to $n.$
\end{proof}
Now we slightly modify the construction of a decreasing sequence $\{ f_k\}$ of positive numbers tending to zero and a subsequence $ \{n_k\}$ given in the case of Frech\'et spaces. Put $ f_1 = e_1$ and $ n_1 =1.$ If $ e_1 \geq 2 e_2,$ then we define $ n_2 =2 $ and $ f_2 = e_2.$
If $ e_1 < 2e_2$ then $ f_2 = \frac{f_1}{2}$ and 
$$
n_2 = \max\{ n \geq 2: f_1 \leq 2e_n\}.
$$
Since $ f_1=e_1 >0$ and $ e_n \rightarrow 0, $ $n_2$ is well-defined.
Now assume that we have constructed positive numbers $ f_1,...,f_k$ and $ n_1,...,n_k\in \mathbb{N}.$ We will construct $ f_{k+1}$ and $ n_{k+1}.$ If 
$$
e_{n_k+1} \leq \frac{f_k}{2}
$$
then we define $ f_{k+1} = e_{n_k+1} $ and $n_{k+1} = n_{k}+1.$ In the opposite case $ f_{k+1} = \displaystyle \frac{f_k}{2}$ and 
$$
n_{k+1} = \max\{ n \geq n_{k}+1: f_k \leq 2e_n\}.
$$ 
Since $ f_k >0$ and $ e_n \rightarrow 0, $ $n_{k+1}$ is well-defined. Reasoning as in Lemma \ref{Konyagin} we can prove:
\begin{lem}
\label{Konyagin1}
Let $ \{ e_n\}$ be a decrasing sequence of positive numbers tending to $0.$ Let $ \{ f_k \} $ and $ \{ n_k\}$ be as in the above construction. Then for any $ k \in \mathbb{N},$ 
$f_k \geq 2f_{k+1}$ and $e_{n_k +1} \leq f_k.$
\end{lem}
Now we can state
\begin{thm}
\label{Banach1}
Let $X$ be a Banach space and let $\{V_n\}$ be a nested sequence of subspaces of $ X$ such that $ \overline{V_n} \subseteq V_{n+1}$ for any $ n \in \mathbb{N}.$ Let $ e_n$ be a decreasing sequence of positive numbers tending to 0.
Then there exists $ x \in X$ such that for any $ n \in \mathbb{N}$
$$
\frac{e_n}{2} \leq dist(x,V_n) \leq 2 e_n.
$$ 
\end{thm}
\begin{proof}
First we show that Theorem \ref{thm:Borodin} remains true under the weaker assumption  
$$ 
e_n \geq \sum_{j=n+1}^{\infty}e_j 
$$ 
for any $n \in \mathbb {N}.$ To prove this modification, for fixed $ k \in \mathbb{N},$ 
and a sequence $ \{ e_n \} $ satisfying the required condition, set $ e_{n,k} = e_n +\displaystyle  \frac{1}{k3^n}.$  It is clear that for any $ n \in \mathbb{N},$ 
$$
e_{n,k} > \sum_{j=n+1}^{\infty}e_{j,k}  \hbox{ for any } n \in \mathbb {N}.
$$
By Theorem \ref{thm:Borodin} there exists $x_k \in X$ such that $ \rho_n{x_k} = e_{n,k}$ for any $ n \in \mathbb{N}.$  Moreover, by the proof of Theorem \ref{thm:Borodin} (see \cite{Br})
$$
x_k = \sum_{n=1}^{\infty} \lambda_{n,k} q_n,
$$ 
where $ q_n \in X, $ $ \| q_n \| \leq 2 $ and $| \lambda_{n,k}| \leq e_{n,k} \leq e_{n,1} $ for any $ n\in \mathbb{N}.$ By the compactness argument and a diagonal process we can assume that there exists a subsequence
$ \{ k_l\} $ such that
$$
\lambda_{n,k_l} \rightarrow_l \lambda_n
$$ 
for any $ n \in \mathbb{N}.$
Notice that for any $ n \in \mathbb{N}$ $ |\lambda_n| \leq e_{n,1}.$ Since the series $ \sum_{n=1}^{\infty}e_{n,1}$ is convergent and $ \|q_n\| \leq 2$ for any $ n \in \mathbb{N}$
the vector 
$$
x = \sum_{n=1}^{\infty} \lambda_n q_n
$$ 
is well-defined and obviousy $ \| x_{k_l} - x\| \rightarrow_l 0.$ Consequently, by the continuity of the function $ \rho_n,$ for any $ n \in \mathbb{N}$
$$
\rho_n(x_{k_l}) = e_{n,k_l} \rightarrow_l e_n = \rho_n(x),
$$
which shows our modification. 
\newline
Now the proof of Theorem \ref{Banach1} proceeds in a similar way as the proof of Theorem \ref{main}. Instead of Theorem \ref{rapidlydecr2}, Lemma \ref{three} and Lemma \ref{Konyagin}, we should apply our modified construction,  
Theorem \ref{thm:Borodin} in our modified version, Lemma \ref{two} and Lemma \ref{Konyagin1}.
\end{proof}
From Theorem \ref{main} we can easily deduce a version of Bernstein's Lethargy Theorem in which subspaces are replaced by sets $ W_n$ satisfying
\begin{equation}
\label{sets}
\overline{span(W_n)} \subseteq W_{n+1}, \hbox{ for } n \in \mathbb{N}.
\end{equation}
\begin{thm}
\label{specialsets}
Let $X$ be a Frech\'et space and let $\{W_n\}$ be a nested sequence of subsets of $ X.$ Set $ V_n = \overline{span(W_n)}$ for $n \in \mathbb{N}.$ Let $\{V_n\} $ satisfy \ref{sets} and the assumptions of Theorem \ref{main}. Suppose that $ e_n$ is a decreasing sequence of positive numbers tending to 0.
Then there exists $ n_o \in \mathbb{N}$ and $ x \in X$ such that for any $ n \geq n_o$
$$
\frac{e_n}{3} \leq dist(x,W_n) \leq 3 e_{n-1}.
$$ 
\end{thm}
\begin{proof}
Notice that by \ref{sets} for any $ x \in X$ and $ n \in \mathbb{N}$
$$
dist(x,V_n) \leq dist(x,W_n) \leq dist(x,V_{n-1}).
$$
By Theorem \ref{main} there exists $x \in X$ and $ n_o \in \mathbb{N}$ such that 
$$
\frac{e_n}{3} \leq dist(x,V_n) \leq 3 e_n.
$$
for $ n \geq n_o,$ which shows our claim. 
\end{proof}
\begin{rem}
If we assume additionally that there exists $M > 0$ such that $\displaystyle \frac{e_n}{e_{n+1}} \leq M$ for any $ n \in \mathbb{N},$ then we obtain a stronger version of Theorem \ref{specialsets}. In this case  there exists $ n_o \in \mathbb{N}$ and $ x \in X$ such that for any $ n \geq n_o$
$$
\frac{e_n}{3} \leq dist(x,W_n) \leq 3M e_{n}.
$$ 
\end{rem}
Notice that the assumption \ref{sets} in Theorem \ref{specialsets} is necessary.
\begin{exm}
Let $X$ be a Banach space and let $ W_n = \{ x \in X: \|x\| \leq n\}.$ Let $ \{e_n\} $ be a decreasing sequence of positive numbers tending to zero. Since $ X =\displaystyle  \bigcup_{n=1}^{\infty} W_n,$ there is no $x\in X$ such that
$$
\frac{e_n}{3} \leq dist(x,W_n) \leq 3 e_{n-1}.
$$ 
\end{exm}
At the end of this section we show some results concerning Fr\'echet spaces  $(X, \| \cdot \|)$ and sequences of subspaces of $X$ $ \mathcal{V}=\{ V_n\}$ satisfying or not satisfying the assumption \ref{nondegenerate}.
First notice that by \cite{Lew}, Prop. 3.4 and Cor. 3.8, if for all $n$, $V_n$ are finite-dimensional, and $X = \overline{\bigcup_{n=1}^{\infty} V_n}$ then $ d_{\mathcal{V}} > 0$ provided 
$ R(\mathcal{V})>0,$ where 
$$
R (\mathcal{V}) = \inf \{ \sup \{ \| tv\| : t \in \mathbb{R}_+ \}, v \in \left(\bigcup_{n=1}^{\infty} V_n\right) \setminus \{0 \} \}.
$$
In general, we have the following:
\begin{lem}
\label{R}
Let $ (X, \| \cdot \|)$ be an infinite-dimensional Fr\'{e}chet space. Let $\mathcal{V} = \{  V_n\} $ be a nested sequence of linear subspaces of $X$ satisfying 
$\overline{V_n} \subseteq V_{n+1}$ such that 
$$
X=\overline{\bigcup_{n=1}^{\infty}V_{n}}.
$$
If $ d_{\mathcal{V}} > 0$ then $ R (\mathcal{V}) >0.$
\end{lem}
\begin{proof}
Notice that for any $ x \in X,$ $ n \in \mathcal{N}$ and $t \in \mathcal{R}_+, $ $ \|tx\| \geq \rho_n(tx).$ Since $X=\overline{\bigcup_{n=1}^{\infty}V_{n}},$ this implies that
$$
R(\mathcal{V}) \geq d_{\mathcal{V}},
$$
which proves our claim.
\end{proof}
\begin{rem}
\label{RdX}
By \cite{Lew}, Prop. 3.4 and Cor. 3.8 and Lemma \ref{R} if all $n$, $V_n$ are finite-dimensional, then $ d_{\mathcal{V}} > 0$ if and only if $ R (\mathcal{V}) >0.$ We do not know if this 
is satisfied for arbitrary $ \mathcal{V}.$
\end{rem}
\begin{lem}
\label{eq}
Let $ (X, \| \cdot \|)$ be an infinite-dimensional Fr\'{e}chet space. Let $\mathcal{V} = \{  V_n\} $ be a nested sequence of linear subspaces of $X$ satisfying 
$\overline{V_n} \subseteq V_{n+1}$ such that 
$$
X=\overline{\bigcup_{n=1}^{\infty}V_{n}}.
$$
Let $ \| \cdot \|_1$ be an F-norm defined on $X$ equivalent to $ \| \cdot \|.$ Denote 
$$
\rho_{n,1}(x) = \inf \{ \| x - v\|_1 : v \in V_n\},
$$
$$
d_{n,\mathcal{V},1} =\sup\{ \rho_{n,1}(x) : x \in V_{n+1}\}
$$
and
$$
d_{\mathcal{V},1} = \inf \{d_{n,\mathcal{V},1}: n \in \mathcal{N}\}.
$$
Then $d_{\mathcal{V}}=0$ if and only if $d_{\mathcal{V},1} =0.$
\end{lem}
\begin{proof}
Assume that $ d_{\mathcal{V}} =0$ and $ d_{\mathcal{V},1} > 0.$ Then there exists $ \epsilon > 0 $ such that $ d_{n,\mathcal{V},1} > \epsilon $ for any $ n \in \mathbb{N}.$ This implies that we can find $ x_n \in V_{n+1}$ such that 
$$
\rho_{n,1}(x_n) \geq \epsilon/2.
$$
On the other hand, since $ d_{\mathcal{V}} = 0,$ $ lim_n \rho_n(x_n) = 0 $ and consequently, there exist $ v_n \in V_n$ such that $\lim_n \| x_n - v_n \| =0.$ Since $ \| \cdot \|_1$ is equivalent to 
$ \| \cdot \| $ 
$$
\epsilon/2 \leq \rho_{n,1}(x_n) \leq \| x_n - v_n \|_1 \rightarrow 0;
$$ 
a contradiction.
\end{proof}
As an application of Lemma \ref{eq} and Theorem \ref{main1} we can state the following theorem.
\begin{thm}
\label{lbounded}
Let $(X, \| \cdot \|)$ be a locally bounded Fr\'echet space and let $\{V_n\}$ be a nested sequence of subspaces of $ X$ such that $ \overline{V_n} \subseteq V_{n+1}$ for any $ n \in \mathbb{N}.$ Let $ e_n$ be a decreasing sequence of positive numbers tending to 0.
Then there exists $ x \in X$ and $ n_o \in \mathbb{N}$ such that for any $ n \geq n_o$
$$
\frac{e_n}{3} \leq dist(x,V_n) \leq 3 e_n.
$$ 
\end{thm}
\begin{proof}
By \cite{Rol}, p.95, Th. 3.2.1, there exists $ 0 < p \leq 1$ and a $p$-homogenous norm equivalent to $ \| \cdot \|.$ By Lemma \ref{eq} and Theorem \ref{main1} we get our result.
\end{proof}
In particular applying Theorem \ref{lbounded} and \cite{Rol}, p. 96, Th. 3.2.2 we get
\begin{thm}
\label{lconvex}
Let $(X, \| \cdot \|)$ be a locally convex Fr\'echet space and let $\{V_n\}$ be a nested sequence of subspaces of $ X$ such that $ \overline{V_n} \subseteq V_{n+1}$ for any $ n \in \mathbb{N}.$ Let $ e_n$ be a decreasing sequence of positive numbers tending to 0.
Then there exists $ x \in X$ and $n_o \in \mathbb{N}$ such that for any $ n \geq n_o$
$$
\frac{e_n}{3} \leq dist(x,V_n) \leq 3 e_n.
$$ 
\end{thm}
Now we state a theorem concerning Fr\'echet spaces in which the topology is determined by a sequence of $ p_n$-homogenous norms.
\begin{thm}
\label{pnnorms} Let $X$ be an infinite dimensional linear space and let $ \{\| \cdot \|_n\}$ be a family of $p_n$-homogenous $F$-pseudonorms which is total over $X,$ i.e for any $x\in X,$ if $ \| x\|_n =0$  for any $n \in \mathbb{N},$
then $x=0.$
Define on $X$ an F-norm $ \| \cdot \|$ by
$$ 
\| x\| = \sum_{j=1}^{\infty} \frac{\|x\|_n}{2^n(1+ \|x\|_n)}.
$$
Assume that $ (X, \| \cdot \|)$ is a Fr\'echet space. Let $\mathcal{V} = \{  V_n\} $ be a nested sequence of linear subspaces of $X$ satisfying 
$\overline{V_n} \subseteq V_{n+1}$ such that 
$$
X=\overline{\bigcup_{n=1}^{\infty}V_{n}}.
$$
Assume that there exists $ N \in \mathbb{N}$ such that for any $ n \in \mathbb{N}$ we can find $x_n \in V_{n+1}$ such that
$$
dist_j (x_n,V_n) = \inf \{ \| x_n -v \|_j : v \in V_n \} > 0
$$
for some $j \in \{ 1,...,N\}.$ 
Then $ d_{\mathcal{V}} > 0.$ 
\end{thm}
\begin{proof}
By our assumptions for any $n \in \mathbb{N}$  and $ t \in \mathbb{R}_+,$ 
$$
\rho_n(tx_n) \geq dist_j(tx_n,V_n)  = t^{p_j} dist_j(x_n, V_n)
$$
for some $j \in \{ 1,...,N\}.$
Hence $ d_{n,\mathcal{V}} \geq \displaystyle \frac{1}{2^N}$ for any $ n \in \mathbb{N}$ which shows that $ d_{\mathcal{V}} \geq \displaystyle \frac{1}{2^N}.$
\end{proof}
\begin{exm}
Let $ \{g_n\} \subset C^{\infty}[0,1]$ be a fixed sequence of orthonormal functions with respect to the scalar product
$ \langle f,g \rangle = \int_0^1 f(t)g(t) dt.$ 
Let 
$$ 
V_n = span[g_1,g_2,..., g_{2n-1}, g_{2n},g_{2n+2},...,g_{2(n+k)}: k \in \mathbb{N}].
$$
Let us equip $X$ with the F-norm $ \| \cdot \|$ given by
$$ 
\| x\| = \sum_{j=1}^{\infty} \frac{\|x\|_n}{2^n(1+ \|x\|_n)},
$$
where $ \| x\|_n = \sup \{ |x^{(n)}(t)|: t \in [0,1] \}.$
Notice that for any $ n \in \mathbb{N},$ and $ v \in V_n,$ 
$$ 
\| g_{2n+1} - v\|_1 \geq (\int_0^1 (g_{2n+1}-v)^2 dt)^{1/2} \geq (\int_0^1 (g_{2n+1})^2 dt)^{1/2} =1.
$$
Hence in this case the assumptions of Theorem \ref{pnnorms} are satisfied with $N=1.$ Consequently, $ d_{\mathcal{V}} > 0$ and by Theorem \ref{main} for any decreasing sequence of positive numbers $ \{ e_n\}$ tending to $0$ there exists $ n_o \in \mathbb{N}$ and $ x \in X$ such that for any $ n \geq n_o$
$$
\frac{e_n}{3} \leq dist(x,V_n) \leq 3 e_n.
$$

\end{exm}


\noindent
\mbox{~~~~~~~}Asuman G\"{u}ven AKSOY\\
\mbox{~~~~~~~}Claremont McKenna College\\
\mbox{~~~~~~~}Department of Mathematics\\
\mbox{~~~~~~~}Claremont, CA  91711, USA \\
\mbox{~~~~~~~}E-mail: aaksoy@cmc.edu \\ \\

\noindent
\mbox{~~~~~~~}Grzegorz LEWICKI\\
\mbox{~~~~~~~}Jagiellonian University\\
\mbox{~~~~~~~}Department of Mathematics\\
\mbox{~~~~~~~}\L ojasiewicza 6, 30-348, Poland\\
\mbox{~~~~~~~}E-mail: Grzegorz.Lewicki@im.uj.edu.pl\\\\


\begin{thebibliography}{99}
\bibitem{Ak-Al} A. G. Aksoy and J. Almira, \textit{On Shapiro's lethargy theorem and some applications}, Jean J. Approx. 6(1), 87-116, (2014).
\bibitem{Ak-Le} A. G. Aksoy and G. Lewicki, \textit{Diagonal operators, s-numbers and Bernstein pairs}.  Note Mat. 17, 209-216 (1999).
\bibitem{Alb} G. Albinus, \textit{Remarks on a theorem of S. N. Bernstein}. Studia Math., 38, 227-234, (1970).
\bibitem{Al-Oi} J. M. Almira and T. Oikhberg, \textit{ Approximation schemes satisfying Shapiro's theorem}, J. Approx. Theory, 534-571 ( 2012). 
\bibitem{Al-To} J. M. Almira and N. del Toro, \textit{ Some remarks on negative results in approximation theory},  Proceedings of the Fourth International Conference on Functional Analysis and Approximation Theory, Vol. I (Potenza, 2000). Rend. Circ. Mat. Palermo (2) Suppl. 2002, no. 68, part I, 245?256. 
\bibitem{Ber} S. N. Bernstein \textit{ On the inverse problem in the theory of best approximation of continuous functions}, Collected works (in Russian), vol II, Izd.  Akad. Nauk, USSR, 1954, pp 292-294.
\bibitem{Br} P. A. Borodin, \textit{ On the existence of an element with given deviations from an expanding system of subspaces}, Mathematical Notes, Vol. 80, No. 5, 621-630, (2006). (Translated from Mathematicheskie Zametki). 
\bibitem{De-Hu} F. Deutsch and H. Hundal, \textit{ A generalization of Tyuremskikh's lethargy theorem  and some applications}, 34(9): 1033-1040, (2013). 
\bibitem{Kon} S. V. Konyagin, \textit{Deviation of elements of a Banach space from a system of subspaces}, Proceedings of the Steklov Institute of Mathematics, 2014, Vol. 284, 204-207, (2014). (Translated from Matematicheskogo Instituta imeni V.A. Steklova).
\bibitem{Kozl} W. M. Koz\l owski, \textit{Modular Function Spaces,} Series in Monographs and Textbooks in Pure and Applied Mathematics, Vol. 122, Dekker, New York, Basel, 1988.
\bibitem{Kozl-Lew}  W. M. Koz\l owski and G. Lewicki, \textit{Analyticity and polynomial approximation in modular function spaces,} Journ. Approx. Theory, 58, (1989), 15 -35.
\bibitem{Lew} G. Lewicki, \textit{Bernstein's ``lethargy" theorem in metrizable topological linear spaces}, Mh. Math. , 113, 213-226, (1992).
\bibitem{Lew1} G. Lewicki, \textit{ A theorem of Bernstein's type for linear projections,} Univ. Iagellon. Acta Math. (27) (1988), 23 - 27. 
\bibitem{Nik} W. N. Nikolski\'{i}, \textit{On some properties of reflexive spaces}, Uchen. Zap. Kalinin. Gos.Ped. Inst., 29, 121-125 (1963) (Russian).
\bibitem{Ple2} W. Ple\'{s}niak, \textit{Quasianalytic functions in the sense of Bernstein}. Dissertationes Math. 147, 1-70, (1977).
\bibitem{Ple1} W. Ple\'{s}niak, \textit{On a theorem of S. N. Bernstein in F-Spaces}, Zeszyty Naukowe Uniwersytetu Jagiellonskiego, Prace Mat. 20, 7-16 (1979).
\bibitem{Pli} A. Plichko, \textit{Rate of decay of the Bernstein numbers,} Zh. Mat. Fiz. Anal. Geom., (9,1) (2013), 59 - 72.
\bibitem{Rol} S. Rolewicz, \textit{Metric Linear Spaces}, Warszawa: PWN, 1982.
\bibitem{Sha} H. S. Shapiro, \textit{Some negative theorems of approximation theory}, Michigan Math. J. 11, 211-217 (1964).
\bibitem{Sin} I. Singer, \textit{Best approximation in normed linear spaces by elements of linear subspaces}, Springer-Verlag, Berlin, 1970.
\bibitem{Tim} A. F. Timan, \textit{Theory of approximation of functions of a real variable} (Russian), Gosudarstv. Izdat. Fiz.-Mat. Lit., Moscow, 1960.
\bibitem{Tyu} I. S. Tyuremskikh, \textit{ On one problem of S. N. Bernstein}, Scientific Proceedings of Kaliningrad State Pedagog. Inst. 52, 123-129 (1967).
\bibitem {Tyu2}  I. S. Tyuremskikh, \textit{The B-Property of Hilbert Spaces},  Uch. Zap. Kalinin. Gos. Pedagog. Inst. 39,  53-64 (1964).
 
\end{thebibliography}
\end{document}